\documentclass[
]{elsarticle}
\usepackage[latin1]{inputenc}
\usepackage[boxed]{algorithm2e}
\usepackage{amsmath}\usepackage{amsthm}\usepackage{amssymb}\usepackage{mathrsfs}\usepackage{amscd}\usepackage{graphicx}\usepackage{subfigure}\usepackage{amsfonts}\usepackage{amsxtra}\usepackage{color}
\usepackage{multirow}
\usepackage{amscd}

\DeclareMathOperator{\PFP}{PFP}
\DeclareMathOperator{\FP}{FP}

\DeclareMathOperator{\supp}{supp}
\theoremstyle{definition}\newtheorem{definition}{Definition}[section]\newtheorem{remark}[definition]{Remark}
\theoremstyle{plain}\newtheorem{theorem}[definition]{Theorem}\newtheorem{lemma}[definition]{Lemma}\newtheorem{corollary}[definition]{Corollary}\newtheorem{proposition}[definition]{Proposition}\newtheorem{conjecture}[definition]{Conjecture}

\newcommand{\R}{\mathbb{R}}\newcommand{\C}{\mathbb{C}}
 
\journal{arXiv.org}
\begin{document}

\begin{frontmatter}

\title{Minimization of the Probabilistic $p$-frame Potential}
\author[a,b]{M.~Ehler\corref{cor1}}
\ead{ehlermar@mail.nih.gov}
\cortext[cor1]{Corresponding author}

\author[b]{K.~A.~Okoudjou}
\ead{kasso@math.umd.edu}
\address[a]{National Institutes of Health, National Institute of Child Health and Human Development, Section on Medical Biophysics, 
Bethesda, MD 20892}
\address[b]{University of Maryland, Department of Mathematics, Norbert Wiener Center,
College Park, MD 20742}

%

%

%


%




\begin{abstract} We investigate the optimal configurations of $n$ points on the unit sphere for a class of potential functions. In particular, we characterize these optimal configurations in terms of their 
approximation properties within frame theory. Furthermore, we consider similar optimal configurations in terms of random distributions of points on the sphere. In this probabilistic setting, we characterize these optimal distributions by means of special classes of probabilistic frames. Our work also indicates some connections between statistical shape analysis and frame theory.
\end{abstract}


\end{frontmatter}

\section{Introduction}

Frames are overcomplete (or redundant) sets of vectors that serve to faithfully represent signals. They were introduced in $1952$ by Duffin and Schaeffer \cite{duscha52}, and reemerged with the advent of wavelets \cite{Christensen:2003ab,  Dau92,Ehler:2007aa, Ehler:2008ab, hw89}. 
Though the overcompleteness of frames precludes signals from having unique representation in the frame expansions, it is, in fact, the driving force behind the use of frames in signal processing \cite{Casazza:2003aa, koche1, koche2}.

In the finite dimensional setting, frames are exactly spanning sets. However, many applications require ``custom-built'' frames that possess additional properties which are dictated by these applications. As a result, the construction of frames with prescribed structures  has been actively pursued. For instance, a special class called {\it finite unit norm tight frames } (FUNTFs) that provide a Parseval-type representation very similar to orthonormal bases, has been customized to model data transmissions \cite{Casazza:2003aa, Goyal:2001aa}. Since then the characterization and construction of FUNTFs and some of their generalizations have received a lot of attention \cite{Casazza:2003aa, koche1, koche2}. Beyond their use in applications, FUNTFs are also related to some deep open problems in pure mathematics such as the Kadison-Singer conjecture \cite{cftw}. FUNTFs appear also in statistics where, for instance,  Tyler used them to construct $M$-estimators of multivariate scatter \cite{Tyler:1987}. We elaborate more on the connection between the $M$-estimators and FUNTFs in Remark~\ref{remark:M estimator FUNTF}. These $M$-estimators were subsequently used to construct maximum likelihood estimators for the the wrapped Cauchy distribution on the circle in \cite{KentTaylor1994} and for the angular central Gaussian distribution on the sphere in \cite{Tyler:1988}. 

FUNTFs are exactly the minimizers of a functional called the frame potential \cite{Benedetto:2003aa}. This was extended to characterize all finite tight frames in \cite{Waldron:2003aa}. Furthermore, in \cite{fjko, jbok}, finite tight frames with a convolutional structure, which can be used to model filter banks, have been characterized as minimizers of an appropriate potential. All these potentials are connected to other functionals whose extremals have long been investigated in various settings. We refer to \cite{cokum06, Delsarte:1977aa, Seidel:2001aa, Venkov:2001aa, Welch:1974aa} for details and related results. 

In the present paper, we study objects beyond both FUNTFs and the frame potential. In fact, we consider a family of functionals, the {\it $p$-frame potentials}, which are defined on  sets $\{x_{i}\}_{i=1}^{N}$ of unit vectors in $\R^d$; see Section~\ref{section:pfp}. These potentials have been studied in the context of spherical $t$-designs for even integers $p$, cf.~Seidel in \cite{Seidel:2001aa}, and their minimizers are not just FUNTFs but FUNTFs that inherit additional properties and structure. Common FUNTFs are recovered only for $p=2$. In the process, we extend Seidel's results on spherical $t$-designs in \cite{Seidel:2001aa} to the entire range of positive real $p$. 

In Section~\ref{section:estimates}, we give lower estimates on the $p$-frame potentials, and prove that in certain cases their minimizers are FUNTFs, which possess additional properties and structure. In particular, if $0<p \leq 2$, we completely characterize the minimizers of the $p$-frame potentials when $N=kd$ for some positive integer $k$. Moreover, when $N=d+1$ and $0<p\leq 2$, we characterize the minimizers of the $p$-frame potentials, under a technical condition, which, we have only been able to establish when $d=2$. We conjecture that this technical condition holds when $d>2$. Finally in Section \ref{section:intro prob}, we introduce {\it probabilistic $p$-frames} that generalize the concepts of frames and $p$-frames. We characterize the minimizers of {\it probabilistic $p$-frame potentials} in terms of probabilistic $p$-frames. The latter problem is solved completely for $0<p\leq 2$, and for all even integers $p$. In particular, these last results generalize \cite{Seidel:2001aa} as well as the recently introduced notion of the probabilistic frame potential in \cite{Ehler:2010aa}.

\smallskip 

\textbf{Further relations to statistics:}
Besides the results on FUNTFs used in \cite{KentTaylor1994,Tyler:1987, Tyler:1988}, and mentioned above, frame theory has essentially evolved independently of statistical fields such as statistical shape analysis \cite{Dryden:1998aa} and directional statistics \cite{Mardia:2008aa}. Nevertheless, there still exist several overlaps, and to the best of our knowledge, these overlaps have not yet been fully explored. Recently, frame theory has been used in directional statistics \cite{Ehler:2010ac}, where FUNTFs are utilized to investigate on statistical tests for directional uniformity and to model and analyze patterns found in granular rod experiments. We must point out that similar results were obtained earlier by Tyler in  \cite{Tyler:1988}. 

Probabilistic tight frames are multivariate probability distributions whose second moments' matrix is a multiple of the identity, and they are used in \cite{Ehler:2010aa} to obtain approximate FUNTFs. The latter approximation procedure is connected to a classical problem in multivariate statistics, namely estimating the population covariance from a sample, which is closely related to the $M$-estimators addressed in \cite{KentTaylor1994,Tyler:1987,Tyler:1988}. The $p$-frame potentials and their probabilistic counterparts that we consider in the sequel, are linked to the notion of shape measure, shape space, and mean shape used in statistical shape analysis. In Section \ref{label:section p frame potential}, we establish a precise connection between the full Procrustes estimate of mean shape \cite[Definition 3.3]{Kent:1994aa} which is the eigenvector corresponding to the largest eigenvalue of the frame operator. Moreover, this eigenvalue coincides with the upper frame redundancy as introduced in \cite{Bodmann:2010aa}. The full Procrustes estimate of mean shape also saturates the upper frame inequality. Moreover, the $p$-frame potentials form size measures as required in statistical shape analysis, and their minimizers among all collections of $N$ points on the sphere define a shape space modulo rotations. 

We hope that the present paper will renew interests in more investigation on the role of frames and the $p$-th frame potential in directional statistics and statistical shape analysis.

 \section{The $p$-frame potential}\label{section:pfp}

\subsection{Background on frames and the frame potential}
To introduce frames and their elementary properties, we follow the textbook \cite{Christensen:2003aa}.
\begin{definition}\label{framedef}
A collection of points $\{x_i\}_{i=1}^N\subset\R^d$ is called a \emph{finite frame for $\R^d$} if there are two constants $0<A\leq B$ such that
\begin{equation}\label{frameineq}
A\|x\|^2 \leq \sum_{i=1}^N |\langle x,x_i\rangle|^2 \leq B\|x\|^2,\quad\text{for all $x\in\R^d$.}
\end{equation}
If the frame bounds $A$ and $B$ are equal, the frame $\{x_i\}_{i=1}^N\subset\R^d$ is called \emph{a finite tight frame for $\R^d$}. In this case, 
\begin{equation}\label{deftightf}
A\|x\|^2 = \sum_{i=1}^N |\langle x,x_i\rangle|^2,\quad\text{for all $x\in\R^d$.}
\end{equation}
A finite tight frame $\{x_i\}_{i=1}^N\subset \R^d$  consisting of unit norm vectors is  called a \emph{finite unit norm tight frame (FUNTF)} for $\R^d$. In this case, the frame bound is $A=N/d$. 

A collection of unit vectors $\{x_i\}_{i=1}^N\subset S^{d-1}$ is called \emph{equiangular} if there exists a nonnegative constant $C$ such that  $|\langle x_i,x_j\rangle|=C$, for $i\neq j$. 
\end{definition}

Given a collection of $N$ points $\{x_i\}_{i=1}^N$ in $\R^d$, the \emph{analysis operator} is the mapping 
 \begin{equation*}
 F: \R^d \rightarrow \R^N, \quad x\mapsto \big(\langle x,x_i \rangle\big)_{i=1}^N.
 \end{equation*} 
 Its adjoint operator is called the \emph{synthesis operator} and given by 
 \begin{equation*}
  F^*: \R^N \rightarrow\R^d , \quad (c_i)_{i=1}^N\mapsto \sum_{i=1}^N c_i x_i.
 \end{equation*}
Using these operators, it is easy to see that $\{x_i\}_{i=1}^N$ is a frame if and only if the \emph{frame operator} defined by
\begin{equation*}
S:= F^* F : \R^d \rightarrow \R^d,\quad x\mapsto \sum_{i=1}^N \langle x,x_i\rangle x_i 
\end{equation*}
is positive, self-adjoint, and invertible. In this case, the following reconstruction formula holds
\begin{equation}\label{eq:reconstruction for frames}
x = \sum_{j=1}^N  \langle S^{-1} x_i,x\rangle x_i,\text{ for all $x\in\R^d$,}
\end{equation}
and $\{S^{-1} x_i\}_{i=1}^N$, in fact, is a frame too, called the \emph{canonical dual frame}. If $\{x_i\}_{i=1}^N$ is a frame, then $\{S^{-1/2}x_i\}_{i=1}^N$ is a finite tight frame. Moreover, note that $\{x_i\}_{i=1}^N$ is a FUNTF if and only if its frame operator $S$ is $\frac{N}{d}$ times the identity.

As mentioned in the introduction, the question of the existence and characterization of FUNTFs was settled in \cite{Benedetto:2003aa}, where the \emph{frame potential}, defined by 
\begin{equation}\label{eq:frame potential}
\FP(\{x_i\}_{i=1}^N) = \sum_{i=1}^N\sum_{j=1}^N |\langle x_i,x_j\rangle |^2,
\end{equation} was introduced and used to give a characterization of its minimizers in terms of FUNTFs. More specifically, they prove the following result:

\begin{theorem}\cite[Theorem 7.1]{Benedetto:2003aa}\label{theorem:Benedetto Fickus}
Let $N$ be fixed and consider the minimization of the frame potential among all collections of $N$ points on the sphere $S^{d-1}$.
\begin{itemize}
\item[a)] If $N\leq d$, then the minimum of the frame potential is $N$. The minimizers are exactly the orthonormal systems for $\R^d$ with $N$ elements.
\item[b)] If $N\geq d$, then the minimum of the frame potential is $\frac{N^2}{d}$. The minimizers are exactly the FUNTFs for $\R^d$ with $N$ elements.
\end{itemize}
\end{theorem}

We shall prove in the sequel that the frame potential is just an example in a family of functionals defined on points on the sphere, and whose minimizers have approximation properties similar to those of the frame potential. But first, we briefly comment on the relation between FUNTFs and $M$-estimators of multivariate scatter:
\begin{remark}\label{remark:M estimator FUNTF}
The concept of FUNTFs is used in signal processing to represent a signal $x\in\R^d$ by means of $x=\frac{d}{N}\sum_{i=1}^N\langle x,x_i\rangle x_i$, which is similar to the well-known expansion in an orthonormal basis. FUNTFs have also been used in statistics to derive $M$-estimators of multivariate scatter: For a sample $\{x_i\}_{i=1}^N\subset \R^d$, where $x_i\neq 0$, for $i=1\ldots,N$, Tyler aims to find a symmetric positive definite matrix $\Gamma$ such that 
\begin{equation*}
M(\Gamma) = \frac{d}{N}\sum_{i=1}^N \frac{\Gamma^{1/2}x_i x_i' \Gamma^{1/2}}{x_i'\Gamma x_i}
\end{equation*}
is the identity matrix. Whenever this is possible, the estimate $V$ of the population scatter matrix is then given by $V=\Gamma^{-1}$. We refer to  \cite{Tyler:1987} for details. Note that $M(\Gamma)=I_d$ implies that 
\begin{equation*}
\bigg\{\frac{\Gamma^{1/2}x_i}{\sqrt{x_i'\Gamma x_i}}\bigg\}_{i=1}^{N}=\bigg\{\frac{\Gamma^{1/2}x_i}{\|\Gamma^{1/2}x_{i}\|}\bigg\}_{i=1}^{N} \subset S^{d-1}
\end{equation*}
 forms a FUNTF. Moreover, $\{x_i\}_{i=1}^N\subset S^{d-1}$ is a FUNTF if and only if $M(I_d)=I_d$.
\end{remark}

\subsection{Definition of the $p$-frame potential}\label{label:section p frame potential}

\begin{definition}\label{pframepotential}
Let $N$ be a positive integer, and $0<p< \infty$. Given a collection of unit vectors $\{x_i\}_{i=1}^{N}\subset S^{d-1}$, the \emph{$p$-frame potential} is the functional
  \begin{equation}\label{eq:pth frame potential}
  \FP_{p, N}(\{x_{i}\}_{i=1}^{N})=\sum_{i, j=1}^{N}|\langle x_i,x_j\rangle |^p. 
 \end{equation}
When, $p=\infty$, the definition reduces to

\begin{equation*}
\FP_{\infty, N}(\{x_{i}\}_{i=1}^{N})=\sup_{i\neq j} |\langle x_i,x_j\rangle |.
\end{equation*}
\end{definition}

It is clear that the $p$-frame potential generalizes the frame potential in \eqref{eq:frame potential}. Finite frames and the $p$-frame potential also extend to complex $\{z_i\}_{i=1}^N\subset S_{\C}^{d-1}=\{z\in\C^d : \|z\|=1\}$ and are related to statistical shape analysis, a tool to quantitatively track the physical deformation of objects. We refer to \cite[Chapters 2, 3 $\&$ 4]{Dryden:1998aa} for more details on shape analysis, but we briefly indicate here its link to the $p$-frame potential.  The shape of an object is specified by landmark points that altogether form the shape space. Often, a suitable transformation is applied first in order to study shape independently on the object's size. To remove the feature of size, we must specify a size measure $g$ (\cite[Definition 2.2]{Dryden:1998aa}), which is a positive function defined on  $S_{\C}^{d-1}$ that satisfies  
\begin{equation*}
g(s\{z_i\}_{i=1}^N) = s g(\{z_i\}_{i=1}^N),\quad\text{for all $s\in\R_{+}$ and $\{z_i\}_{i=1}^N\subset S_{\C}^{d-1}$}.
\end{equation*}
It is immediate that the following family of functionals can be seen as size measures: 
\begin{equation*}
g_p(\{z_i\}_{i=1}^N) := \big(\sum_{i, j=1}^{N}|\langle z_i,z_j\rangle |^{p}\big)^{\frac{1}{p}} =\bigg(\FP_{N,p}(\{z_i\}_{i=1}^N)\bigg)^{1/p}.
\end{equation*}
In particular, $g_1$ represents the centroid size (when the shape is centered at the origin), which is one of the most common size measures in statistical shape analysis. Given two complex configurations $z_1,z_2\in\C^d$ derived from landmarks that code two-dimensional shape, the full Procrustes distance (\cite[Definition 3.2]{Dryden:1998aa})  is 
\begin{equation*}
d_F(z_1,z_2)= \inf_{\beta,\theta,a,b} \bigg{\|}\frac{z_1}{\|z_1\|} - \frac{z_2}{\|z_2\|}\beta e^{i\theta} -a -ib \bigg{\|}.
\end{equation*}
 Given $N$ configurations $\{z_i\}_{i=1}^N\subset \C^d$, the full Procrustes estimate of mean shape is defined by 
\begin{equation*}
\bar{z}_P = \arg\inf_{\|z\|=1} \big( \sum_{i=1}^N d_F^2(z_i,z) \big),
\end{equation*}
cf.~\cite[Definition 3.3]{Dryden:1998aa}. The average axis from the complex Bingham maximum likelihood estimator is the same as the full Procrustes estimate of mean shape for two-dimensional shapes, and the same holds for the complex Watson distribution, cf.~\cite[Sections 6.2 $\&$ 6.3]{Dryden:1998aa}. If we assume that $\{z_i\}_{i=1}^N$ are centered around zero, then $\bar{z}_P$ is given by the eigenvector corresponding to the largest eigenvalue $\lambda$ of the frame operator of the normalized collection $\{\frac{z_i}{\|z_i\|}\}_{i=1}^N$. Furthermore, one observes that this eigenvalue $\lambda$ is the upper frame redundancy of $\{z_i\}_{i=1}^N$ as introduced in \cite{Bodmann:2010aa}, which also coincides with the optimal upper frame bound $B$ in \eqref{frameineq}. Therefore, the full Procrustes estimate of mean shape satisfies the upper frame inequality with equality. Moreover, we have observed that the root mean square of the full Procrustes estimate of mean shape is $1-\frac{\lambda}{N}$.

Before stating the next elementary result, we recall some basic definitions in physics: A {\it conservative force} $F$, is a vector field defined on $\R^d$, such that $-F$ is the gradient of some {\it potential} $P$ that is then induced by the conservative force. Lemma~\ref{conservative} below was proved for the frame potential in \cite{Benedetto:2003aa}, and we extend it to all $p$-frame potentials with $1<p<\infty$:
\begin{lemma}\label{conservative}
For each $p\in (1, \infty)$, the $p$-frame potential $\FP_{p, N}$ is induced by the conservative force $F_p= f_p(\|a-b\|)(a-b)$, for $a,b\in S^{d-1}$, where
\begin{equation*}
 f_p(x):=\begin{cases}
 p(1-\frac{x^2}{2})^{p-1},&\text{for } |x|\leq \sqrt{2},\\
 -p(\frac{x^2}{2}-1)^{p-1},&\text{otherwise.}
 \end{cases}
 \end{equation*}
\end{lemma}
$F_p$ is a central force between the `particles' $a$ and $b$ that we call the \emph{$p$-frame force}.
\begin{proof}   The function  
 \begin{equation*}
 \mathbf{p}_p(x):=|1-\frac{x^2}{2}|^p=\begin{cases}
 (1-\frac{x^2}{2})^p,&\text{for } |x|\leq \sqrt{2},\\
 (\frac{x^2}{2}-1)^p,&\text{otherwise,}
 \end{cases}
 \end{equation*}
is differentiable and satisfies  $ \mathbf{p}_p'(x)=-xf_p(x)$. This is sufficient to verify that the potential $P_p(a,b):=\mathbf{p}_p(\|a-b\|)$, defined for $a,b\in S^{d-1}$, satisfies $\nabla_a P_p(a,b) = -F(a,b)$, where $b$ is held fixed. Thus, $F_p$ is a conservative vector field. The physical meaningful potential $P_p(a,b)$ is in fact given by 
$$ P_p(a,b)  =\mathbf{p}_p(\|a-b\|) = \big |1-\frac{1}{2}\|a-b\|^2 \big |^p = \big|\langle a,b\rangle\big|^p,$$
 where we used that $\|a\|=\|b\|=1$. 
 Consequently, the $p$-frame potential is induced by the conservative central force $F_p$. 
 \end{proof}
 
 As a consequence of the above lemma, $\{x_i\}_{i=1}^N\subset S^{d-1}$ are in \emph{equilibrium} under the $p$-frame force if they minimize the $p$-frame potential among all collections of $N$ points on the sphere. Note that such a collection of equilibria modulo rotations form a shape space. 

\begin{remark}
We will use repeatedly the fact that for a fixed $\{x_{i}\}_{i=1}^{N} \subset S^{d-1}$, the $p$-frame potential $\FP_{p,N}(\{x_{i}\}_{i=1}^{N})$ is a decreasing and continuous  function of $p \in (0, \infty)$. 

\end{remark}

\section{Lower estimates for the $p$-frame potential}\label{section:estimates}
We start this section with a few elementary results about the minimizers of the $p$-frame potential as well as their connection to $t$-designs. In fact, potentials on the sphere, and $t$-designs have been well investigated \cite{Bachoc:2005aa, cokum06, Delsarte:1977aa, Seidel:2001aa, Venkov:2001aa}. However, one of the key differences between $t$-designs and our $p$-frame potential is that the former is considered only for positive integers $t$ while the latter is investigated for $p\in (0, \infty)$.  

\subsection{The Welch bound revisited}\label{subsection:Welch}
If $p=2k$ is an even integer, one can use Welch's results~\cite{Welch:1974aa} to conclude that, for $\{x_i\}_{i=1}^N\subset S^{d-1}$,
\begin{equation}\label{eq:Welch}
\FP_{p, N}(\{x_{i}\}_{i=1}^{N}) \geq \frac{N^2}{\binom{d+k-1}{k}}.
\end{equation}
We shall verify that this estimate is not optimal for small $N$, by proving an  estimate for $\FP_{p, N}$ when $2<p<\infty$. The following Proposition first appeared in \cite{Oktay:2007}: 
\begin{proposition}\label{prop:potential for p>2}
Let $\{x_i\}_{i=1}^N\subset S^{d-1}$, $N\geq d$, and $2<p<\infty$, then
\begin{equation}\label{eq:potential for p>2}
\FP_{p, N}(\{x_{i}\}_{i=1}^{N}) \geq N(N-1)\big(\frac{N-d}{d(N-1)}\big)^{p/2}+N,
\end{equation}
 and equality holds if and only if  $\{x_i\}_{i=1}^N$ is an equiangular FUNTF.
\end{proposition}

\begin{proof}
For $\frac{1}{2}=\frac{1}{p}+\frac{1}{r}$, H\"older's inequality yields
\begin{equation}\label{eq:hoelder numer 1}
\|(\langle x_i,x_j\rangle)_{i\neq j}\|_{\ell_2} \leq \|(\langle x_i,x_j\rangle)_{i\neq j}\|_{\ell_p}(N(N-1))^{1/r}.
\end{equation}
Raising to the $p$-th power and applying $\frac{1}{r}=\frac{1}{2}-\frac{1}{p}$ leads to 
\begin{equation}\label{eq:raising to the power}
\|(\langle x_i,x_j\rangle)_{i\neq j}\|_{\ell_2}^p \leq \|(\langle x_i,x_j\rangle)_{i\neq j}\|_{\ell_p}^p(N(N-1))^{p/2-1}.
\end{equation}
Therefore, 
$$
 \sum_{i\neq j}|\langle x_i,x_j\rangle |^p  \geq  \big(\sum_{i\neq j}|\langle x_i,x_j\rangle |^2\big)^{p/2} (N(N-1))^{1-p/2}.$$
Using the fact that $\sum_{i\neq j}|\langle x_i,x_j\rangle |^2\geq \frac{N^2}{d}-N$ (see Theorem~\ref{theorem:Benedetto Fickus}) implies that
$$ \sum_{i\neq j}|\langle x_i,x_j\rangle |^p \geq \big(N(\frac{N}{d}-1)\big)^{p/2} (N(N-1))^{1-p/2} = N(N-1)\bigg(\frac{N-d}{d(N-1)}\bigg)^{p/2},$$
which proves~\eqref{eq:potential for p>2}. 

To establish the last part of the Proposition, we recall that an equiangular FUNTF $\{x_{k}\}_{k=1}^{N} \subset \R^d$ satisfies 
\begin{equation}\label{eq:equi}
|\langle x_i,x_j\rangle | = \sqrt{\frac{N-d}{d(N-1)}},\quad \text{ for all }i\neq j
\end{equation} 
see, \cite{Casazza:2008ab,Sustik:2007aa}, for details. Consequently, if $\{x_{k}\}_{k=1}^{N}$ is an equiangular FUNTF, then~\eqref{eq:potential for p>2} holds with equality. 

On the other hand, if equality holds in \eqref{eq:potential for p>2}, then $\sum_{i\neq j}|\langle x_i,x_j\rangle |^2 = \frac{N^2}{d}-N$ and $\{x_i\}_{i=1}^N$ is a FUNTF due to Theorem \ref{theorem:Benedetto Fickus}. Moreover, the H\"older estimate \eqref{eq:hoelder numer 1} must have been an equality which means that $|\langle x_i,x_j\rangle|=C$ for $i\neq j$, and some constant $C\geq 0$. Thus, the FUNTF must be equiangular.
\end{proof}

 By comparing \eqref{eq:Welch} with \eqref{eq:potential for p>2}, it is easily seen that the Welch bound is not optimal for small $N$:

 \begin{proposition}\label{prop:second one}
 Let $\{x_i\}_{i=1}^N\subset S^{d-1}$ and $p=2k>2$ be an even integer. If $d<N \leq  \binom{d+k-1}{k}$, then 
  \begin{equation}\label{eq:abc}
 \FP_{p,N}(\{x_{k}\}_{k=1}^{N}) \geq N(N-1)\big(\frac{N-d}{d(N-1)}\big)^k +N >\frac{N^2}{\binom{d+k-1}{k}}.
 \end{equation} 
 \end{proposition}

 \begin{proof} 
 The condition on $N$ implies $1 \geq   \frac{N}{\binom{d+k-1}{k}}$, and adding $ (N-1)\big(\frac{N-d}{d(N-1)}\big)^k>0$ to the right hand side leads to
 \begin{equation*}
  (N-1)\big(\frac{N-d}{d(N-1)}\big)^k +1 >  \frac{N}{\binom{d+k-1}{k}}.
 \end{equation*}
 Multiplication by $N$ and Proposition \ref{prop:potential for p>2} then yield \eqref{eq:abc}.
 \end{proof}

\begin{remark} The estimate in Proposition \ref{prop:potential for p>2} is sharp if and only if an equiangular FUNTF exists. In \cite[Sections 4 $\&$ 6]{Sustik:2007aa}, construction (and hence existence) of equiangular FUNTFs was established when $d+2 \leq N \leq 100$. For general $d$ and $N$, a necessary condition for existence of equiangular FUNTFs is given, and it is conjectured that the conditions are sufficient as well. The authors essentially provide on upper bound on $N$ that depends on the dimension $d$. 
Therefore, Proposition~\ref{prop:potential for p>2} might not be optimal when the redundancy $N/d$ is much larger than $1$. 
\end{remark}

\subsection{Relations to spherical $t$-designs}\label{subsection:t design}
A \emph{spherical $t$-design} is a finite subset $\{x_i\}_{i=1}^N$ of the unit sphere $S^{d-1}$ in $\R^d$,
such that,
\begin{equation*}
\frac{1}{N}\sum_{i=1}^N h(x_i) = \int_{S^{d-1}} h(x)d\sigma(x),
\end{equation*}
for all homogeneous polynomials $h$ of total degree equals or less than $t$ in $d$ variables and where $\sigma$ denotes the uniform surface measure on $S^{d-1}$ normalized to have mass one. The following result is due to \cite[Theorem 8.1]{Venkov:2001aa} (see \cite{Seidel:2001aa}, \cite{Delsarte:1977aa} for similar results).

 \begin{theorem}\label{theorem:p even integer discrete}\cite[Theorem 8.1]{Venkov:2001aa}
Let $p=2k$ be an even integer and $\{x_i\}_{i=1}^N=\{-x_i\}_{i=1}^N\subset S^{d-1}$, then 
  \begin{equation*}
\FP_{p, N}(\{x_{i}\}_{i=1}^{N})  \geq  \frac{1\cdot 3\cdot 5\cdots(p-1)}{d(d+2)\cdots (d+p-2) }N^2,
 \end{equation*}
and equality holds if and only if $\{x_i\}_{i=1}^N$ is a spherical $p$-design. 
 \end{theorem}

\subsection{Optimal configurations for  the $p$-frame potential}  
We first use Theorem~\ref{theorem:Benedetto Fickus} to characterize the minimizers of the $p$-frame potential for $0<p<2$ provided that the number of points $N$ is a multiple of the dimension $d$:

\begin{theorem}\label{prop:k multiples and small p}
Let $0<p<2$ and assume that $N=kd$ for some positive integer $k$. Then the minimizers of the $p$-frame potential are exactly the $k$ copies of any orthonormal basis modulo multiplications by $\pm 1$. The minimum of~\eqref{eq:pth frame potential}, over all sets of $N=kd$ unit norm vectors, is $k^{2}d$.
\end{theorem}
\begin{proof}
If we fix a collection of vectors $\{x_i\}_{i=1}^N$, then the frame potential is a decreasing function in $p\in(0, 2)$. Therefore, $$N^{2}/d=\FP_{2,N}(\{x_{i}\}_{i=1}^{N}) \leq \FP_{p,N}(\{x_{i}\}_{i=1}^{N}).$$

Now suppose that $\{y_{i}\}_{i=1}^{N}$ consists of $k$ copies of an  orthonormal basis $\{e_{i}\}_{i=1}^{d}$for $\R^d$. Then,  $$\min \FP_{p, N}(\{x_{i}\}_{i=1}^{N}) \leq \FP_{p, N}(\{y_{i}\}_{i=1}^{N})=\FP(\{y_{i}\}_{i=1}^{N})=N^{2}/d.$$ 
Consequently, $$\min \FP_{p, N}(\{x_{k}\}_{k=1}^{N})=\min \FP(\{x_{k}\}_{k=1}^{N})=N^{2}/d.$$ 
Thus, $\FP_{p,N}$ has the same minimum as $\FP$. Clearly $k$ copies of an orthonormal basis of $\R^d$ form a  FUNTF, and hence minimize the $\FP$ due to Theorem \ref{theorem:Benedetto Fickus}. Consequently, the $k$ copies of an orthonormal basis also minimize the $p$-frame potential  for $0<p<2$. On the other hand, the inner products must be $0$ or $1$ to obtain a minimizer. The smallest $p$-frame potential then have the $k$ copies of an orthonormal basis.
\end{proof}

We now consider the $p$-frame potential for $N=d+1$. The case $p=2$ is covered by Theorem \ref{theorem:Benedetto Fickus}. Note that any FUNTF with $N=d+1$ vectors is equiangular~\cite{Goyal:2001aa,Strohmer:2003aa}. Hence, the case $2<p<\infty$ is settled by Proposition~\ref{prop:potential for p>2},  so we focus on $p\in (0,2)$. 

One easily verifies that, for $p_0=\frac{\log(\frac{d(d+1)}{2})}{\log(d)}$, an orthonormal basis plus one repeated vector and an equiangular FUNTF have the same $p_0$-frame potential $\FP_{p_{0}, d+1}$. Under the assumption that those two systems are exactly the minimizers of $\FP_{p_{0}, d+1}$, the next result will give a complete characterization of the minimizers of $\FP_{p, d+1}$, for $0<p<2$. However, we have only been able to establish the validity of this assumption when $d=2$, cf.~Corollary \ref{theorem:3 points}.
\begin{theorem}\label{theorem:d+1}
Let $N=d+1$ and set $p_{0}=\frac{\log(\frac{d(d+1)}{2})}{\log(d)}= \frac{\log(\frac{N(N-1)}{2})}{\log(N-1)}$. Let $\{x_i\}_{i=1}^N\subset S^{d-1}$, and assume that $\FP_{p_{0},N}(\{x_{i}\}_{i=1}^{N}) \geq N+2,$ with equality holds if and only if $\{x_i\}_{i=1}^N$ is an orthonormal basis plus one repeated vector or an equiangular FUNTF. Then,
\begin{itemize}
\item[\textnormal{(1)}] for $0<p<p_{0} $, then for any $\{x_i\}_{i=1}^N\subset S^{d-1}$, we have  $\FP_{p,N}(\{x_{i}\}_{i=1}^{N}) \geq N+2,$ and equality holds if and only if $\{x_i\}_{i=1}^N$ is an orthonormal basis plus one repeated vector,
\item[\textnormal{(2)}] for $p_{0}<p<2$, then for any $\{x_i\}_{i=1}^N\subset S^{d-1}$, we have $\FP_{p,N}(\{x_{i}\}_{i=1}^{N}) \geq 2^{\frac{p}{p_{0}}}\, (Nd)^{1-\frac{p}{p_{0}}} + N,$ 
and equality holds if and only if $\{x_i\}_{i=1}^N$ is an equiangular FUNTF.
\end{itemize}
\end{theorem}

\begin{proof}
Under the assumptions of the Theorem, let $0<p<p_0$, then 
\begin{equation*}
\min \FP_{p_{0}, N}(\{x_{i}\}_{i=1}^{N}) \leq \min \FP_{p, N}(\{x_{i}\}_{i=1}^{N}),
\end{equation*}
and using $\{y_{i}\}_{i=1}^{N}$ consisting of an orthonormal basis of $\R^d$ with one repeated vector, yields that 
\begin{equation*}
\min \FP_{p, N}(\{x_{i}\}_{i=1}^{N})\leq \FP_{p, N}(\{y_{i}\}_{i=1}^{N})=N+2=\min \FP_{p_{0}, N}(\{x_{i}\}_{i=1}^{N}). 
\end{equation*}
Consequently, $\min \FP_{p_{0}, N}(\{x_{i}\}_{i=1}^{N}) = \min \FP_{p, N}(\{x_{i}\}_{i=1}^{N})=N+2$. Since an orthonormal basis plus one repeated vector minimizes the $p$-frame potential for $p=p_{0}$, it must also minimize $\FP_{p, N}$ for  $0<p<p_{0}$, which proves $(1)$.

Assume now that $p_0 < p < 2$. Choose $r$ such that $\frac{1}{p_{0}}=\frac{1}{p}+\frac{1}{r}$, H\"older's  inequality yields
\begin{equation*}
\|(\langle x_i,x_j\rangle)_{i\neq j}\|_{\ell_{p_{0}}}^p \leq \|(\langle x_i,x_j\rangle)_{i\neq j}\|_{\ell_p}^p(N^{2}-N)^{p/r}.
\end{equation*}
Since we assume that $(2)$ holds, we have 
$ \sum_{i\neq j}|\langle x_i,x_j\rangle |^{p_{0}} \geq  2 $, which leads to 
\begin{align*}
 \sum_{i\neq j}|\langle x_i,x_j\rangle |^p & \geq  \bigg(\sum_{i\neq j}|\langle x_i,x_j\rangle |^{p_{0}}\bigg){\frac{p}{p_{0}}} (N^{2}-N)^{1-\frac{p}{p_{0}}}\\
 &\geq 2^{\frac{p}{p_{0}}}\, (Nd)^{1-\tfrac{p}{p_{0}}}.
\end{align*}
This concludes the proof of $(2)$. 
By applying \eqref{eq:equi}, one then checks that an equiangular FUNTF satisfies $(2)$ with equality.

The ``only'' part comes from the fact that the H\"older inequality becomes an equality  only if the sequences are linearly dependent. This means that the $\{x_i\}_{i=1}^N$ are equiangular. They must then satisfy $|\langle x_i,x_j\rangle | = \frac{1}{d}$. Thus, by~\eqref{eq:equi}  they form an equiangular FUNTF \cite{Casazza:2008ab,Sustik:2007aa}.
\end{proof}

When $d=2$ we can in fact verify the main hypothesis of Theorem~\ref{theorem:d+1}, which leads to the following result:

\begin{corollary}\label{theorem:3 points}
Let $\{x_i\}_{i=1}^3\subset S^{1}$, and set $p_{0}=\frac{\log(3)}{\log(2)}$. Then, $$\FP_{p_{0},3}(\{x_{i}\}_{i=1}^{3}) \geq 5,$$ and equality holds if and only if $\{x_i\}_{i=1}^3$ is an orthonormal basis plus one repeated vector or an equiangular FUNTF. 

Consequently,
\begin{itemize}
\item[\textnormal{(1)}] for $0<p<p_{0} $, then for any $\{x_i\}_{i=1}^3\subset S^{1}$, we have $\FP_{p,3}(\{x_{i}\}_{i=1}^{3}) \geq 5,$ and equality holds if and only if $\{x_i\}_{i=1}^3$ is an orthonormal basis plus one repeated vector,
\item[\textnormal{(2)}] for $p_{0}<p<\infty$, then for any $\{x_i\}_{i=1}^3\subset S^{1}$, we have $\FP_{p,3}(\{x_{i}\}_{i=1}^{3}) \geq \frac{6}{2^{p}} +3,$ 
and equality holds if and only if $\{x_i\}_{i=1}^3$ is an equiangular FUNTF.
\end{itemize}
\end{corollary}
The minimum of $\FP_{p,3}$, for $0<p<\infty$ is plotted in Figure \ref{figure:min of 2d 3 vectors}. 
 \begin{figure}
 \centering
  \includegraphics[width=.35\textwidth]{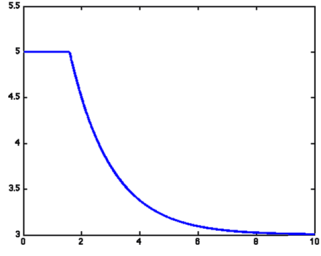}\hspace{1ex}
 \caption{Minimum of $\FP_{p,3}$ in Corollary \ref{theorem:3 points}, for $0<p<10$.}\label{figure:min of 2d 3 vectors}
 \end{figure}
\begin{proof}
Clearly  $(1)$ and $(2)$ follows from Theorem~\ref{theorem:d+1} once the minimizers of $\FP_{p_{0}, N}$ are characterized.  

Without loss of generality, let $\beta$ be the smallest angle between $x_1$, $x_2$, and $x_3$ and let $\alpha$ be the second smallest angle between them. This yields, of course, $0\leq \beta\leq \alpha$.

Case 1: For $0\leq  \alpha+\beta\leq\frac{\pi}{2}$, we have 
\begin{equation*}
\frac{1}{2}\sum_{i\neq j}|\langle x_i,x_j\rangle|^p = \cos^p(\alpha) +\cos^p(\beta)+\cos^p(\alpha+\beta)=F(\alpha, \beta).
\end{equation*} 
Since $1<p_{0}$, the $p$-frame potential is differentiable in $\alpha$ and $\beta$, and its critical points are 
\begin{align*}
0 & =-p\cos^{p-1}(\alpha)\sin(\alpha)-p\cos^{p-1}(\alpha+\beta)\sin(\alpha+\beta)\\
0 & = -p\cos^{p-1}(\beta)\sin(\beta)-p\cos^{p-1}(\alpha+\beta)\sin(\alpha+\beta).
\end{align*}
This implies that either $\alpha=\beta=0$ or $\beta=0$ and $\alpha=\frac{\pi}{2}$. In the first case, we have a maximum since it implies $x_1=x_2=x_3$. The latter case means that two points are identical and the third one is perpendicular which is a potential minimum of the $p$-frame potential. 

Case 2: We can assume that $\frac{\pi}{4}\leq \alpha$ (otherwise we are in Case 1). We can further assume that $\frac{\pi}{4}\leq \alpha \leq \frac{\pi}{2}\leq \alpha+\beta\leq \frac{2\pi}{3}$ (if $\frac{2\pi}{3}<\alpha+\beta\leq\pi$. Otherwise, substitute $x_i$ with $-x_i$.
We now have 
\begin{equation*}
\frac{1}{2}\sum_{i\neq j}|\langle x_i,x_j\rangle|^p =\cos^p(\alpha) +\cos^p(\beta)+(-\cos(\alpha+\beta))^p=G(\alpha, \beta).
\end{equation*} 
The critical points of $G$ are given by
\begin{align*}
0 & =-p\cos^{p-1}(\alpha)\sin(\alpha)+p(-\cos(\alpha+\beta))^{p-1}\sin(\alpha+\beta)\\
0 & = -p\cos^{p-1}(\beta)\sin(\beta)+p(-\cos(\alpha+\beta))^{p-1}\sin(\alpha+\beta).
\end{align*}
By subtracting one equation from the other and raising to the second power, we obtain
 \begin{equation*}
 z_1^{p-1}(1-z_1) = z_2^{p-1}(1-z_2),
 \end{equation*}
 where $z_1=\cos(\alpha)^2$ and $z_2=\cos^2(\alpha+\beta)$. Since $\sin^2(x)-1/2\geq \cos^2(x)$ for all $\pi/2\leq x\leq 2\pi/3$, we have $0\leq z_1\leq 1/2$ and $0\leq z_2\leq 1/4$ because $z_2\leq 1-z_2-1/2$.
 
 We consider the function $F(z)=z^{p-1}(1-z)$ on $0\leq z\leq 1/2$. $F$ achieves its maximum at $z=\frac{p-1}{p}\approx 0.3691$ and is convex, cf.~Figure \ref{figure:1}.%
 \begin{figure}
 \centering
 \subfigure[$F(z)=z^{p-1}(z-1)$, $0\leq z\leq \frac{1}{2}$]{\hspace{1ex}
 \includegraphics[width=.35\textwidth]{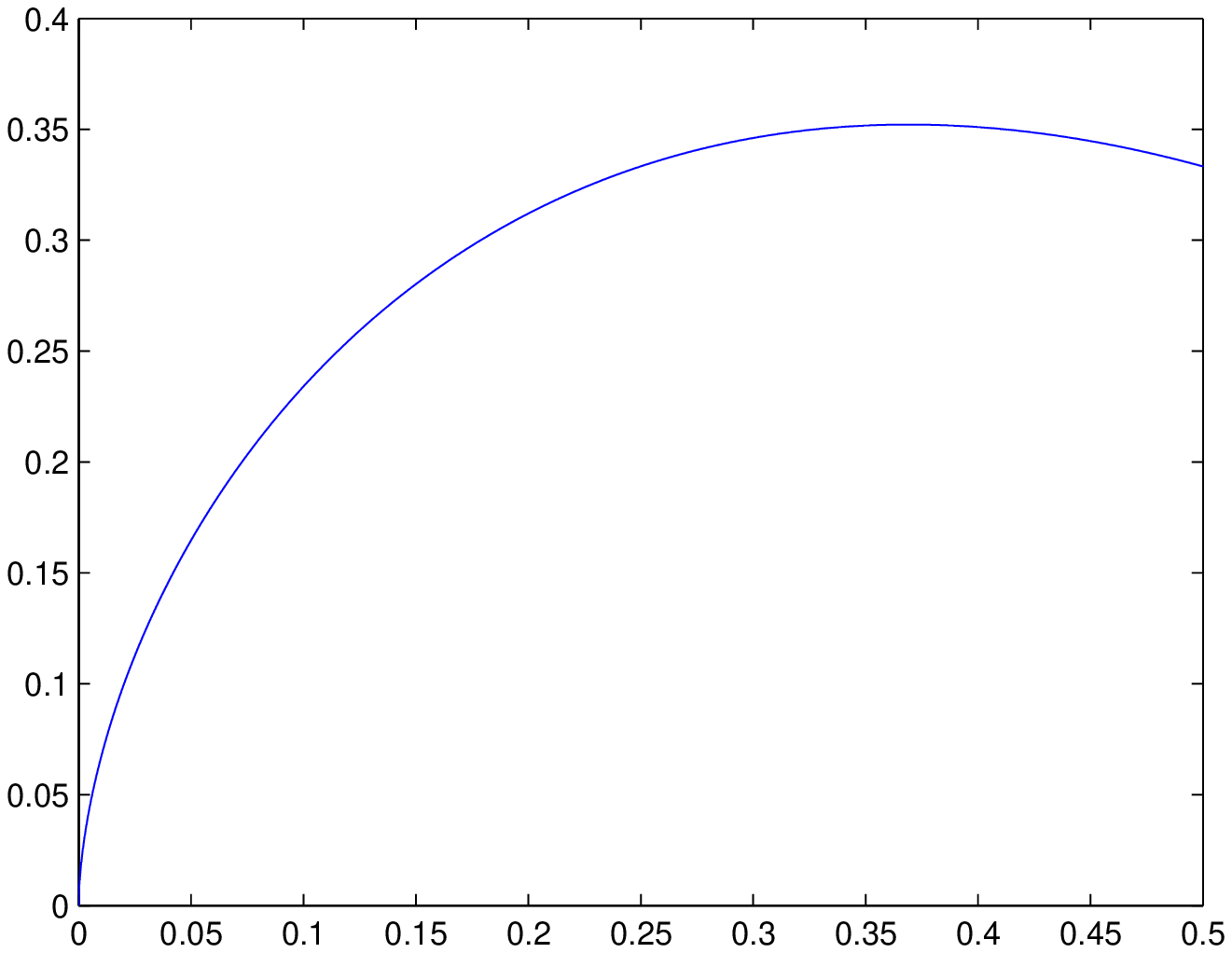}\hspace{1ex}}
 \hspace{2ex}
 \subfigure[$f(\alpha) =  2\cos^p(\alpha) + (-\cos(2\alpha))^p$, $\frac{\pi}{3}\leq \alpha\leq \frac{\pi}{2}$]{\hspace{7ex}
 \includegraphics[width=.35\textwidth]{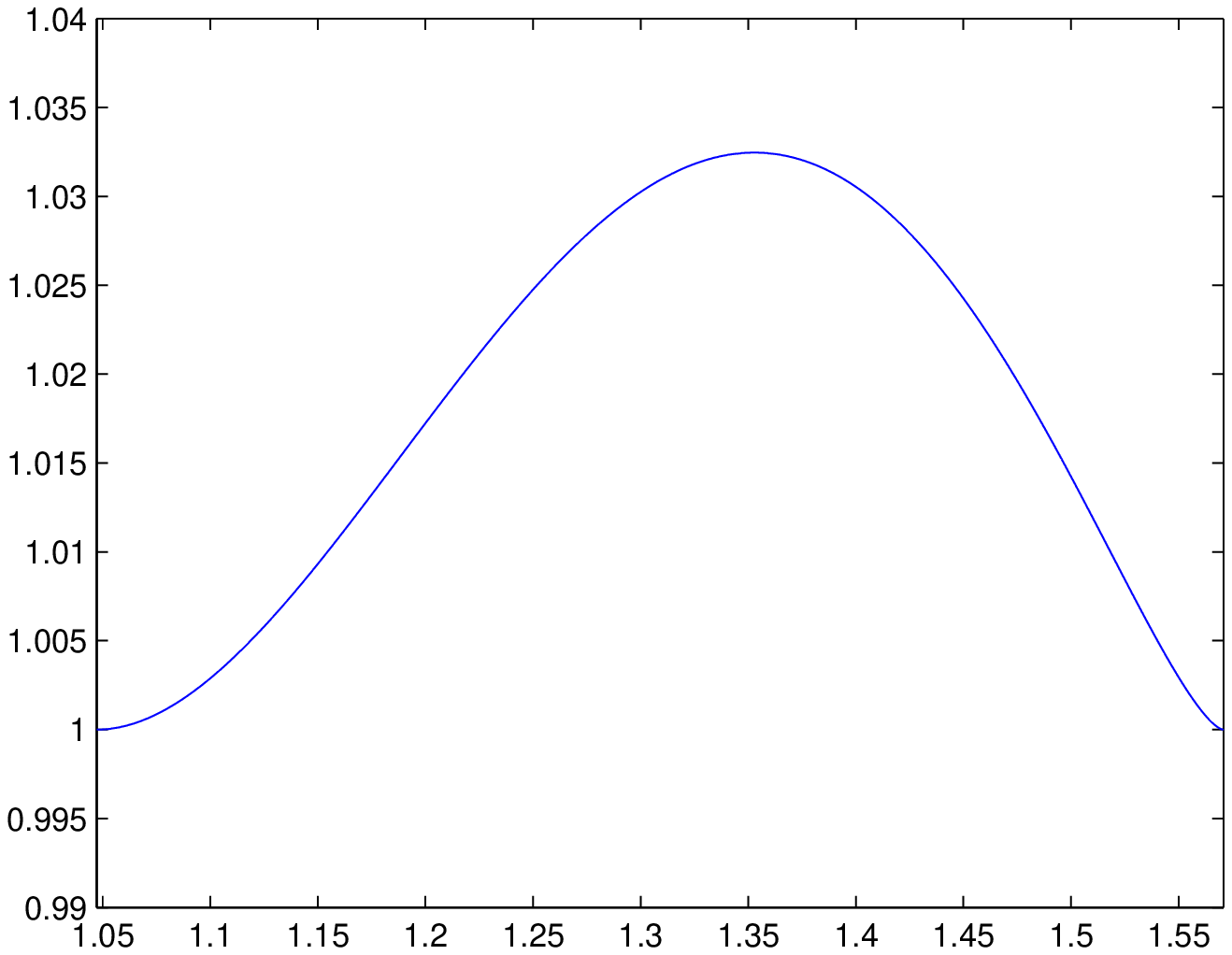}\hspace{7ex}}
 \caption{$F$ and $f$ in the proof of Theorem \ref{theorem:3 points}}\label{figure:1}
 \end{figure}
 Therefore, for $0\leq z_1\leq 1/2$ and $0\leq z_2\leq 1/4$, we have $F(z_1)=F(z_2)$ if and only if $z_1=z_2$ or $z_1=1/2$ and $z_2=1/4$. For $z_1=\cos^2(\alpha)=1/2$, we have $\alpha=\pi/3$ and $z_2=\cos^2(\pi/3+\beta)=1/4$ yields $\beta=\pi/3$. The case $z_1=z_2$ leads to $\cos^2(\alpha)=\cos^2(\alpha+\beta)$ which implies $\alpha+\beta-\pi/2 = \pi/2-\alpha$. This is equivalent to $\beta=\pi-2\alpha$. Since we assume $\beta\leq \alpha$, we obtain $\pi/3\leq \alpha \leq \pi/2$. 

We now check the minima of the function
 \begin{align*}
f(\alpha)  & =  \cos^p(\alpha)+\cos^p(\pi-2\alpha)+(-\cos(\pi-\alpha))^p\\
& = 2\cos^p(\alpha) + (-\cos(2\alpha))^p,
 \end{align*}
 on $\pi/3\leq\alpha\leq \pi/2$, cf.~Figure \ref{figure:1}. Its derivative 
 \begin{equation*}
 \frac{\partial f}{\partial \alpha} (\alpha) = -2p\cos^{p-1}(\alpha)\sin(\alpha) + 2p(-\cos(2\alpha))^{p-1}\sin(2\alpha)
 \end{equation*}
vanishes if and only if 
 \begin{equation*}
 \cos^{p-1}(\alpha)\sin(\alpha) = (\sin^2(\alpha)-\cos^2(\alpha))^{p-1}2\cos(\alpha)\sin(\alpha)
 \end{equation*}
 For $\alpha\neq \pi/2$, this yields
  \begin{equation*}
 \cos^{p-1}(\alpha) = (1-2\cos^2(\alpha))^{p-1}2\cos(\alpha).
 \end{equation*}
 By substituting $x=\cos(\alpha)$, we obtain, for $0<x\leq 1/2$
  \begin{equation*}
 x^{p-1} = (1-2x^2)^{p-1}2x,
 \end{equation*}
which is equivalent to 
  \begin{equation*}
 1/2 = (1/x-2x)^{p-1}x \Leftrightarrow
 0 = (1/x-2x)x^{1/(p-1)} -(1/2)^{1/(p-1)}
 \end{equation*}
and define the new function 
   \begin{equation*}
 g(x) = -2x^{q+1}+x^{q-1} -(1/2)^{q},
 \end{equation*}
 where $q=1/(p-1)$. To show that $g$ has only one extremal point on $0<x\leq 1/2$, we differentiate
   \begin{equation*}
  \frac{\partial g}{\partial x}(x) = -2(q+1)x^q+(q-1)x^{q-2}.
 \end{equation*}
 The term $ \frac{\partial g}{\partial x}(x)$ vanishes if and only if 
 \begin{equation*}
 (q-1)x^{q-2} = 2(q+1)x^q,
 \end{equation*}
which is equivalent to
 \begin{equation*}
 x^2 = \frac{q-1}{2(q+1)} =\frac{2-\log_2(3)}{2\log_2(3)}.
 \end{equation*}
 Hence, $x\approx 0.3618$ and $ \frac{\partial g}{\partial x}$ does not have any other zeros on $0<x\leq 1/2$. This means that $g$ has only one extremal point and can then only have two zeros on $0<x\leq 1/2$. The zero at $x=1/2$ corresponds to a minimum. This means that the zero of $ \frac{\partial g}{\partial x}$ at $x\approx 0.3618$ is a maximum of $g$. Hence the other zero of g is between $0$ and $\approx 0.3618$. However, this other zero corresponds to a maximum of $f$. The minimum of $f$ can thus be at $x=0$ or $x=1/2$. This implies $\alpha=\pi/3$ or $\alpha=\pi/2$. It is easy to verify that $\alpha=\pi/3$ would lead to $\beta=\pi/3$, and $\alpha=\pi/2$ yields $\beta=0$. Thus, the minimum of the $p$-frame potential corresponds to either an orthonormal basis plus one repeated element ($\alpha=\pi/2$, $\beta=0$) or an equiangular FUNTF ($\alpha=\beta=\pi/3$). One easily checks that both situations lead to the same global minimum.
 
\end{proof}

In view of Theorem \ref{theorem:d+1} and Corollary \ref{theorem:3 points}, we have the following conjecture: 
\begin{conjecture}\label{conjecture:p_0}
Let $N=d+1$  and $p_0=\frac{\log(\frac{d(d+1)}{2})}{\log(d)}$. Then $$\FP_{p_{0},N}(\{x_{k}\}_{k=1}^{N}) \geq N+2,$$ and equality holds if and only if $\{x_i\}_{i=1}^N$ is an orthonormal basis plus one repeated vector or an equiangular FUNTF.
\end{conjecture}

One can check that $1<p_0<2$, for $d>1$. According to Proposition \ref{prop:potential for p>2}, the minimizers of the $p$-frame potential for $2<p<\infty$ are exactly the equiangular FUNTFs. Thus, our conjecture essentially addresses the range $0<p<2$. 

\begin{remark}
Although we are primarily interested in real unit norm vectors $\{x_i\}_{i=1}^N\subset \R^d$, we should mention that Theorem \ref{theorem:Benedetto Fickus}, Theorem~\ref{prop:k multiples and small p}, Equation \eqref{eq:Welch}, and Propositions \ref{prop:potential for p>2} and \ref{prop:second one} still hold for complex vectors $\{z_i\}_{i=1}^N\subset \C^d$ that have unit norm. The constraints on $N$ and $d$ that allow for the existence of a complex FUNTF are slightly weaker than in the real case~\cite{Sustik:2007aa}.
\end{remark}

\section{The probabilistic $p$-frame potential}\label{section:intro prob}

The present section is dedicated to introducing a probabilistic version of the previous section. We shall consider probability distributions on the sphere rather than finite point sets. Let $\mathcal{M}(S^{d-1},\mathcal{B})$ denote the collection of probability distributions on the sphere with respect to the Borel sigma algebra $\mathcal{B}$. 

We begin by introducing the probabilistic $p$-frame which generalizes the notion of probabilistic frames introduced in~\cite{Ehler:2010aa}. 

\begin{definition}\label{probpframe}
For $0<p<\infty$, we call $\mu\in\mathcal{M}(S^{d-1},\mathcal{B})$ a \emph{probabilistic $p$-frame} for $\R^d$ if and only if there are constants $A,B>0$ such that
 \begin{equation}\label{ppframeineq}
A\|y\|^p\leq  \int_{S^{d-1}} |\langle x,y\rangle|^p d\mu(x) \leq B\|y\|^p, \quad\forall y\in\R^d.
 \end{equation}
 We call $\mu$ a \emph{tight probabilistic $p$-frame} if and only if we can choose $A=B$.
 \end{definition}
 Due to Cauchy-Schwartz, the upper bound $B$ always exists. 
Consequently, in order to check that $\mu$ is a probabilistic $p$-frame one only needs to focus on the lower bound $A$. 

 Since the uniform surface measure $\sigma$ on $S^{d-1}$ is invariant under orthogonal transformations, one can easily check that it constitutes a tight probabilistic $p$-frame, for any $0<p<\infty$.

Given a probability measure $\mu \in \mathcal{M}(S^{d-1},\mathcal{B})$, we call 
\begin{equation*}
F : \R^d \rightarrow L_p(S^{d-1},\mu),\quad x\mapsto \langle x,\cdot\rangle 
\end{equation*}
the \emph{analysis operator}. It is trivially seen that $$\|F(y)\|_{L_{p}(S^{d-1}, \mu)}=\|\langle x,y\rangle\|_{L_{p}(S^{d-1}, \mu)} \leq \|y\|$$ for all $0<p\leq \infty$. The dual of $F$ is called \emph{synthesis operator} and is given by
\begin{equation*}
F^* : L_q(S^{d-1},\mu) \rightarrow \R^d,\quad f\mapsto \int_{S^{d-1}} f(x)xd\mu(x),
\end{equation*}
where $1=\frac{1}{p}+\frac{1}{q}$, and $1\leq p \leq \infty$. In fact, $F^*$ is well-defined and bounded operator on all $L_{r}(S^{d-1}, \mu)$ where 
$1\leq r \leq \infty$. Indeed, for  $f \in L_{r}(S^{d-1}, \mu)$ we have 
 $$
\|F^{*}(f)\|\leq \int_{S^{d-1}}|f(x)|\|x\|d\mu(x)= \int_{S^{d-1}}|f(x)|d\mu(x)\leq \|f\|_{L_{r}}.$$ 
Given $\mu \in \mathcal{M}(S^{d-1},\mathcal{B})$, the second moments matrix of $\mu$ is the $d\times d$ matrix defined by 
\begin{equation}\label{ppfoperator}
S=F^* F = \bigg(\int_{S^{d-1}} x_i x_j d\mu(x) \bigg)_{i,j}
\end{equation}
with respect to the canonical basis for $\R^d$. As we show next, the second moments matrix $S$ plays a key role in determining if $\mu$ is a probabilistic $p$-frame. We refer to \cite{Christensen:2003ab}, where a  result was proved for similar discrete $p$-frames.

\begin{proposition}\label{pfonto}
If $1\leq p <\infty$, then $\mu \in \mathcal{M}(S^{d-1},\mathcal{B})$ is a probabilistic $p$-frame if and only if $F^*$ is onto. 
\end{proposition}

\begin{proof}
As mentioned earlier the upper bound in the probabilistic $p$-frame definition always holds. So we only need to show the equivalence between the lower bound and the surjectivity of $F^{*}$. 

Assume that $F^{*}$ is surjective. Since $F$ is injective, then $S$ is invertible and for each $y \in \R^d$ we have
$$\|y\|^{2}=|\langle Sy,S^{-1}y\rangle_{\R^{d}}|=|\langle F^{*}Fy,S^{-1}y\rangle_{\R^{d}}|=|\langle Fy,FS^{-1}y\rangle_{L^{p}\to L^{p'}}|,$$ which can be estimated as follows:
$$\|y\|^{2} \leq \|F(y)\|_{L^{p}}\|F(S^{-1} y)\|_{L^{p'}}\leq C \|F(y)\|_{L^{p}}\|S^{-1} (y)\|_{\R^{d}}\leq C \|F(y)\|_{L^{p}}\|y\|_{\R^{d}}.$$
Therefore, for all $y\neq 0 \in \R^d$ we have $1/C \|y\| \leq \|F(y)\|_{L^{p}}$ that is $$ 1/C \|y\|^{p}\leq \int_{S^{d-1}}|\langle x,y\rangle|^p d\mu(x).$$ Thus $\mu$ is a probabilistic $p$-frame.

Now assume that $\mu$ is a probabilistic $p$-frame, but that $F^{*}$ is not surjective. Then, there exists $z \neq 0 \in \R^d$ such that $\langle z, F^{*}f\rangle=0$ for all $f \in L_{p'}(S^{d-1}, \mu)$. Consequently $$\langle z, F^{*}f\rangle_{\R^{d}}=\langle F(z), f\rangle_{L_{p}\to L_{p'}}=\int_{S^{d-1}}f(x)\langle x , z\rangle, d\mu(x)=0$$ for all $f \in L_{p'}$. A contradiction argument leads to $\langle x, z\rangle =0$ for all $x \in S^{d-1}$ which implies that $z=0$. Thus $F^{*}$ is surjective.
\end{proof}
 
The second moments matrix can also be  used to show that a probabilistic $p$-frame gives rise to a reconstruction formula that extends the finite frame expansion in \eqref{eq:reconstruction for frames}. In addition, the next result generalizes the reconstruction formula for tight probabilistic frames obtained in \cite[Lemma 3.7]{Ehler:2010aa}.

\begin{proposition}\label{reconppf} Let $0<p<\infty$ and assume that $\mu \in \mathcal{M}(S^{d-1},\mathcal{B})$ is a probabilistic $p$-frame for $\R^d$. Set $\tilde{\mu}=\mu \circ S$. Then for each $y \in \R^d$, we have
\begin{equation}\label{reconsppf}
y= \int_{S^{-1}(S^{d-1})}Sz \, \langle z, y\rangle \, d\tilde{\mu}(z)= \int_{S^{-1}(S^{d-1})}z \, \langle Sz, y\rangle \, d\tilde{\mu}(z).
\end{equation} 
\end{proposition}

\begin{proof}
The result follows by noticing that $y=SS^{-1}y=S^{-1}Sy$.
\end{proof}

The above result motivates the following definition:

\begin{definition}\label{candualppf} Let $0< p < \infty$. If $\mu \in \mathcal{M}(S^{d-1},\mathcal{B})$ is a probabilistic $p$-frame with frame operator $S$, then $\tilde{\mu}=\mu \circ S \in \mathcal{M}(S^{-1}(S^{d-1}),S^{-1}\mathcal{B})$ is called the \emph{probabilistic canonical dual frame} of $\mu$
\end{definition}

Note that if $\mu$ in Proposition \ref{reconppf} is the counting measure corresponding to a FUNTF $\{x_i\}_{i=1}^N$, then $\tilde{\mu}$ is the counting measure associated to the canonical dual frame of $\{x_i\}_{i=1}^N$.


\begin{lemma}\label{equivappf} a) If $\mu$ is probabilistic frame, then it is a probabilistic $p$-frame for all $1\leq p< \infty$. Conversely, if $\mu$ is a probabilistic $p$-frame for some $1\leq p < \infty$, then it is a probabilistic frame.

\noindent b) Let $1\leq p < \infty$. If $\mu$ is a probabilistic $p$-frame, then so is the canonical dual $\tilde{\mu}$.
\end{lemma}

\begin{proof}
a) 
Assume that $\mu$ is a probabilistic frame and let $1\leq p < \infty$. Then, we only need to check that the lower inequality of~\eqref{ppframeineq} holds, since the corresponding upper bound is trivial. 
By Proposition~\ref{pfonto} (applied to $p=2$), $S=F^{*}F$ is invertible and for each $y \in \R^d$ we have
$$\|y\|^{2}=|\langle Sy,S^{-1}y\rangle_{\R^{d}}|=|\langle F^{*}Fy,S^{-1}y\rangle_{\R^{d}}|=|\langle Fy,FS^{-1}y\rangle_{L^{p}\to L^{p'}}|,$$ which can be estimated as follows:
$$\|y\|^{2} \leq \|F(y)\|_{L_{p}}\|F(S^{-1} y)\|_{L_{p'}}\leq C \|F(y)\|_{L_{p}}\|S^{-1} (y)\|_{\R^{d}}\leq C \|F(y)\|_{L_{p}}\|y\|_{\R^{d}}.$$ Therefore, for all $y\neq 0 \in \R^d$ we have $1/C \|y\| \leq \|F(y)\|_{L_{p}}$ that is $$ 1/C \|y\|^{p}\leq \int_{S^{d-1}}|\langle x,y\rangle|^p d\mu(x).$$ 
For the converse, assume that $\mu$ is a probabilistic $p$-frame for some $p>2$. Then, for all $y\neq 0 \in \R^d$, 
\begin{align*}
A\|y\|^p&\leq  \int_{S^{d-1}} |\langle x,y\rangle|^p d\mu(x)\\
&=\int_{S^{d-1}} |\langle x,y\rangle|^2\, |\langle x,y\rangle|^{p-2}\, d\mu(x)\\
&  \leq \int_{S^{d-1}} \|x\|^{p-2}\, \|y\|^{p-2}\, |\langle x,y\rangle|^2 d\mu(x) \\
&= \|y\|^{p-2}\,  \int_{S^{d-1}} |\langle x,y\rangle|^2 d\mu(x),
\end{align*}from which it follows that $$A\|y\|^2 \leq \int_{S^{d-1}} |\langle x,y\rangle|^2 d\mu(x).$$

If $\mu$ is a probabilistic $p$-frame for some $p<2$. Then, for all $y\neq 0 \in \R^d$, 
$$\|y\|^{2}=|\langle Sy,S^{-1}y\rangle_{\R^{d}}|=|\langle F^{*}Fy,S^{-1}y\rangle_{\R^{d}}|=|\langle Fy,FS^{-1}y\rangle_{L_{p}\to L_{p'}}|,$$ which can be estimated by  
$$\|y\|^{2} \leq \|Fy\|_{L_{p}}\|FS^{-1}y\|_{L_{p'}}\leq C \|Fy\|_{L_{2}}\|y\|,$$ where we have used the fact that for $p<2$, $L_{2}(S^{d-1}, \mu) \subset L_{p}(S^{d-1}, \mu)$. This conclude the proof of a).

b) If $\mu$ is a probabilistic $p$-frame for some $1\leq p < \infty,$ then by a) $\mu$ is a probabilistic frame. In this case, $\tilde{\mu}$ is known to be a probabilistic frame, cf.~\cite{Ehler:2010aa}, and thus a probabilistic $p$-frame. 
\end{proof}

We are particularly interested in tight probabilistic $p$-frame potentials, which we seek to characterize in terms of minimizers of appropriate potentials. This motivates the following definition: 

 \begin{definition}\label{profframpot}
For $0<p<\infty$ and $\mu\in\mathcal{M}(S^{d-1},\mathcal{B})$, the \emph{probabilistic $p$-frame potential} is defined by  
  \begin{equation}\label{eq:prob p pot}
 \PFP(\mu,p) = \int_{S^{d-1}}\int_{S^{d-1}} |\langle x,y\rangle|^p d\mu(x) d\mu(y).
 \end{equation}
\end{definition}

From the weak-star-compactness of the collection of all probability distributions on the sphere, we can deduce that $\PFP(\mu, p)$ admits a minimizer which satisfies
\begin{equation}\label{minppfp}
\PFP(p) = \min_{\mu\in\mathcal{M}(S^{d-1},\mathcal{B})} \PFP(\mu,p).
\end{equation}

We now turn to the minimizers of the probabilistic frame potential $\PFP(\mu)$. In the process, we extend some ideas developed  in \cite{Bjoerck:1955aa} to the probabilistic frame potential. 
\begin{proposition}\label{prop:1}
Let $0<p<\infty$ and let $\mu$ be a minimizer of \eqref{eq:prob p pot}, then 
\begin{itemize}
\item[\textnormal{(1)}] $\int_{S^{d-1}} |\langle x,y\rangle|^p d\mu(x) = \PFP(p)$, for all $y\in\supp(\mu)$,
\item[\textnormal{(2)}] $\int_{S^{d-1}} |\langle x,y\rangle|^p d\mu(x)\geq \PFP(p)$, for all $y\in S^{d-1}$.
\end{itemize}
\end{proposition}
\begin{proof}
The proof will use the following observation. Let  $\mu$ be a probability measure on $S^{d-1}$ and choose a measure  $\nu$, such that $\nu(S^{d-1})=0$ and $\mu+\varepsilon \nu\geq 0$, for all $0\leq \varepsilon\leq 1$. Let us also introduce the notation $$ \PFP(\mu,\nu,p):=\int_{S^{d-1}}\int_{S^{d-1}} |\langle x,y\rangle|^p d\mu(x) d\nu(y).$$ We then obtain
\begin{align*}
\PFP(\mu) & \leq \PFP(\mu+\varepsilon \nu,p)\\
& = \PFP(\mu,p)+\varepsilon^2\PFP(\nu,p) + 2\varepsilon \PFP(\mu,\nu,p)\\
& = \PFP(\mu)+\varepsilon^2\PFP(\nu,p) + 2\varepsilon \PFP(\mu,\nu,p).
\end{align*} 
We thus have $0\leq \varepsilon \PFP(\nu,p)+2\PFP(\mu,\nu,p)$, for all $0\leq \varepsilon\leq 1$, which implies $\PFP(\mu,\nu,p)\geq 0$. 

We now prove $(1)$ using a contradiction argument. In particular, assume that $(1)$ does not hold. This implies that there are $y_1, y_2\in \supp(\mu)$ such that 
\begin{equation*}
a := \int_{S^{d-1}} |\langle x,y_2\rangle|^p d\mu(x) < \int_{S^{d-1}} |\langle x,y_1\rangle|^p d\mu(x) =: b.
\end{equation*}
Set $P_\mu(y)=\int_{S^{d-1}} |\langle x,y\rangle|^p d\mu(x)$. 
Let $K$ be an open ball around $y_1$ in $S^{d-1}$ and so small that $y_2\not\in K$ and that the oscillation of $P_\mu(y)$ on $K$ is smaller than $\frac{b-a}{2}$. Let $m=\mu(K)>0$. One can check that the measure $\nu$ defined by 
\begin{equation*}
\nu(E) := m\delta_{y_2}(E)-\mu(E\cap K),\quad E\in\mathcal{B},
\end{equation*}
satisfies $\nu(S^{d-1})=0$, and $\mu +\epsilon \nu \geq 0$. Hence, $\PFP(\mu,\nu,p)\geq 0$. On the other hand, we can estimate
$$
\PFP(\mu,\nu,p)  = \int_{S^{d-1}} P_\mu(y)d\nu(y)= P_\mu(y_2)m - \int_K P_\mu(y)d\mu(y)= am - \int_K P_\mu(y)d\mu(y)$$ and so

$$\PFP(\mu,\nu,p)  \leq am - (b-\frac{b-a}{2})m  = - \frac{b-a}{2}m <0.$$

This is a contradiction to $\PFP(\mu,\nu,p)\geq 0$ and implies that there is a constant $C$ such that $P_\mu(y)=C$, for all $y\in\supp(\mu)$.  
We still have to verify that the constant $C$ is in fact $\PFP(p)$: 
\begin{align*}
\PFP(p)  = \PFP(\mu,p) & =  \int_{S^{d-1}} P_\mu(y)d\mu(y) \\
& =  \int_{\supp(\mu)} P_\mu(y)d\mu(y)\\
& =  \int_{\supp(\mu)} C d\mu(y) = C.
\end{align*}

The proof of $(2)$ is similar to the one above, and so we omit it. 
\end{proof}
The following result is an immediate consequence of Proposition~\ref{prop:1}. 

\begin{corollary}\label{theorem:tight p frame is necessary}
Let $0<p<\infty$ and let $\mu$ be a minimizer of \eqref{eq:prob p pot}, then
\begin{itemize}
\item[\textnormal{(1)}] $\int_{S^{d-1}} |\langle x,y\rangle|^p d\mu(x) = \PFP(p)\|y\|^p$, for all $\frac{y}{\|y\|}\in\supp(\mu)$,
\item[\textnormal{(2)}] $\supp(\mu)$ is a complete subset of $\R^d$. 
\end{itemize}
\end{corollary}
\begin{proof}
$ (1)$ directly follows from $(1)$ in Proposition \ref{prop:1}.

$(2)$ If $\supp(\mu)$ is not complete in $\R^d$, then there is an element $y\in S^{d-1}$ that is in the orthogonal complement. However, this contradicts (2) in Proposition \ref{prop:1}. 
\end{proof}

We can now characterize the minimizers of the probabilistic $p$-frame potential when  $0<p<2$. In fact, we shall show that these minimizers are discrete probability measures, and the following theorem is the analogue of Proposition \ref{prop:k multiples and small p}:
 \begin{theorem}\label{theorem:p less than 2}
 Let $0<p<2$, then the minimizers of  \eqref{eq:prob p pot} are exactly those probability distributions $\mu$ that satisfy both, 
 \begin{itemize}
 \item[(i)] there is an orthonormal basis $\{x_1,\ldots,x_d\}$ for $\R^d$ such that 
 \begin{equation*}
 \{x_1,\ldots,x_d\} \subset \supp(\mu) \subset \{\pm x_1,\ldots,\pm x_d\},
 \end{equation*}
 \item[(ii)] there is $f:S^{d-1}\rightarrow \R$ such that $\mu(x) = f(x) \nu_{\pm x_1,\ldots,\pm x_d}(x)$ and
 \begin{equation*}
 f(x_i) +f(-x_i) = \frac{1}{d}.
 \end{equation*} 
 \end{itemize} 
 \end{theorem}
 The measure $\nu_{\pm x_1,\ldots,\pm x_d}(x)$ in Theorem \ref{theorem:p less than 2} denotes the counting measure of the set $\{\pm x_i : i=1,\ldots,d\}$.

 \begin{proof}
 Since $0\leq |\langle x,y\rangle|\leq 1$, for $x,y\in S^{d-1}$, we have $\PFP(\mu,p)\geq \PFP(\mu,2)$. In ~\cite[Theorem 3.10]{Ehler:2010aa} it was shown that the normalized counting measure $\frac{1}{d}\nu_{x_1,\ldots,x_d}$ of an orthonormal basis minimizes $\PFP(\cdot,2)$. Due to $\PFP(\frac{1}{d}\nu_{x_1,\ldots,x_d},2) = \PFP(\frac{1}{d}\nu_{x_1,\ldots,x_d},p)$, we obtain that $\frac{1}{d}\nu_{x_1,\ldots,x_d}$ also minimizes $\PFP(\cdot,p)$ and hence $\PFP(p)=\PFP(2)$. 
 
 In the following, we prove that all minimizers of $\PFP(\cdot,p)$ are essentially induced by an orthonormal basis. Let $\mu$ be a minimizer and let $v,w\in \supp(\mu)$. We first show that $|\langle v,w\rangle |\in \{0,1\}$. The implications $\langle v,w\rangle =1$ if and only if $v=w$ and $\langle v,w\rangle =-1$ if and only if  $v=-w$ are trivial. 
 
 Suppose now that $v\neq \pm w$ and $\langle v,w\rangle \neq 0$, then there exist  $\varepsilon>0$ and $\delta_\varepsilon>0$ such that 
 \begin{itemize}
 \item[(a)] $B_\varepsilon(v)\cap B_\varepsilon(w)=\emptyset$ and $\mu(B_\varepsilon(v)), \mu(B_\varepsilon(w)) \geq \delta_\varepsilon$. 
 \item[(b)] for all $x\in B_\varepsilon(v)$ and $y\in B_\varepsilon(w)$, $|\langle x,y\rangle |^p\geq |\langle x,y\rangle |^2+\varepsilon$.
 \end{itemize}
By using $B=B_\varepsilon(v)\times B_\varepsilon(w)$, this implies
\begin{align*}
\PFP(\mu,p) & = \int_{B} |\langle x,y\rangle|^p d\mu(x) d\mu(y) + \int_{S^{d-1}\times S^{d-1}\setminus B } |\langle x,y\rangle|^p d\mu(x) d\mu(y)\\
& \geq \int_{B} (|\langle x,y\rangle|^2+\varepsilon) d\mu(x) d\mu(y) + \int_{S^{d-1}\times S^{d-1}\setminus B } |\langle x,y\rangle|^2 d\mu(x) d\mu(y)\\
& = \PFP(\mu,2) + \varepsilon \mu(B_\varepsilon(v)) \mu(B_\varepsilon(w))\\
&\geq \PFP(\mu,2) +\varepsilon \delta_\varepsilon^2 > \PFP(\mu,2),
\end{align*}
 which is a contradiction. Thus, we have verified that $|\langle x,y\rangle|\in \{0,1\}$, for all $x,y\in \supp(\mu)$. Distinct elements in $\supp(\mu)$ are then either orthogonal to each other or antipodes. According to Corollary \ref{theorem:tight p frame is necessary}, $\supp(\mu)$ is complete in $\R^d$. Thus, there must be an orthonormal basis $\{x_i\}_{i=1}^d$ such that
 \begin{equation*}
  \{x_1,\ldots,x_d\} \subset \supp(\mu) \subset \{\pm x_1,\ldots,\pm x_d\}.
 \end{equation*} 
Consequently, there is a density $f:S^{d-1}\rightarrow\R$ that vanishes on $S^{d-1}\setminus \supp(\mu)$ such that $\mu(x)=f(x)\nu_{\pm x_1,\ldots,\pm x_d}(x)$.  
 
To verify that $f$ satisfies (ii), let us define $\tilde{f}:S^{d-1}\rightarrow \R$ by 
\begin{equation*}
\tilde{f}(x)=\begin{cases} f(x)+f(-x),& x\in\{x_1,\ldots,x_d\}\\
0,& \text{ otherwise. } 
\end{cases}
\end{equation*}
This implies that $\tilde{\mu}(x)=\tilde{f}(x)\nu_{x_1,\ldots,x_d}(x)$ is also a minimizer of $\PFP(\cdot,2)$. But the minimizers of the  probabilistic frame potential for $p=2$ have been investigated in~\cite[Section 3]{Ehler:2010aa}. We can follow the arguments given there to obtain $\tilde{f}(x_i)=\frac{1}{d}$, for all $i=1,\ldots,d$. 
\end{proof}

 For even integers $p$, we can  give the minimum of $\PFP(\mu, p)$ and characterize its minimizers. The following theorem generalizes Theorem \ref{theorem:p even integer discrete}. Moreover, note that the bounds are now sharp, i.e., for any even integer $p$, there is a probabilistic tight $p$-frame: 
 
 \begin{theorem}\label{theorem:p even integer}
 Let $p$ be an even integer. For any probability distribution $\mu$  on $S^{d-1}$,  
  \begin{equation*}
 \PFP(\mu, p)=\int_{S^{d-1}}\int_{S^{d-1}} |\langle x,y\rangle|^p d\mu(x) d\mu(y) \geq  \frac{1\cdot 3\cdot 5\cdots(p-1)}{d(d+2)\cdots (d+p-2) },
 \end{equation*}
and equality holds if and only if $\mu$ is a probabilistic tight $p$-frame. 
 \end{theorem}

 \begin{proof}
 Let $\alpha=\frac{d}{2}-1$ and consider the Gegenbauer polynomials $\{C_{n}^{\alpha}\}_{n\geq 0}$ defined by 
 \begin{equation*}
 C_0^\alpha(x)  = 1, \qquad C_1^\alpha(x) = 2 \alpha x,
 \end{equation*}
 \begin{align*}
C_{n}^\alpha(x) &= \frac{1}{n}[2x(n+\alpha-1)C_{n-1}^\alpha(x) - (n+2\alpha-2)C_{n-2}^\alpha(x)]\\
&= C_n^{(\alpha)}(z)=\sum_{k=0}^{\lfloor n/2\rfloor} (-1)^k\frac{\Gamma(n-k+\alpha)}{\Gamma(\alpha)k!(n-2k)!}(2z)^{n-2k}.
\end{align*}
$\{C_{n}^{(\alpha)}\}_{n=1}^s$ is an orthogonal basis for the collection of polynomials of degree less or equal to $s$ on the interval $[-1,1]$ with respect to the weight
\begin{equation*}
w(z) = \left(1-z^2\right)^{\alpha-\frac{1}{2}},
\end{equation*} 
i.e., for $m\neq n$,
 \begin{equation*}
 \int_{-1}^1 C_n^{(\alpha)}(x)C_m^{(\alpha)}(x)w(x)\,dx = 0.
 \end{equation*}
 They are normalized by
 \begin{equation*}
 \int_{-1}^1 \left[C_n^{(\alpha)}(x)\right]^2(1-x^2)^{\alpha-\frac{1}{2}}\,dx = \frac{\pi 2^{1-2\alpha}\Gamma(n+2\alpha)}{n!(n+\alpha)[\Gamma(\alpha)]^2}.
 \end{equation*}
The polynomials $t^p$, $p$ an even integer, can be represented by means of
\begin{equation*}
t^p=\sum_{k=0}^p \lambda_k C^{\alpha}_k(t).
\end{equation*}
It is known (see, e.g.,~\cite{Bachoc:2005aa,Delsarte:1977aa}) that $\lambda_i> 0$, $i=0,\ldots,p$, and $\lambda_0$ is given by
 \begin{equation*}
\lambda_0= \frac{1}{c}\int_{-1}^1 t^p w(t) dt,
 \end{equation*}
 where 
 \begin{equation*}
 c = \frac{\pi 2^{d+3}\Gamma(d-2) }{(\frac{d}{2}-1)\Gamma(\frac{d}{2}-1)^2}.
 \end{equation*}
 Moreover,  $C^\alpha_k$ induces a positive kernel, i.e., for $\{x_i\}_{i=1}^N\subset S^{d-1}$ and $\{u_i\}_{i=1}^N\subset \R$,
 \begin{equation*}
 \sum_{i,j=1}^{N} u_iC^{\alpha}_k (\langle x_i,x_j\rangle )u_j \geq  0, \quad \forall k=0,1,2,...
 \end{equation*}
 see~\cite{Bachoc:2005aa,Delsarte:1977aa}. Note that the probability measures with finite support are weak star dense in $\mathcal{M}(S^{d-1},\mathcal{B})$. Since $C^{\alpha}_k$ is continuous, we obtain, for all $\mu\in \mathcal{M}(S^{d-1},\mathcal{B})$, 
 \begin{equation*} 
 \int_{S^{d-1}} \int_{S^{d-1}} C^{\alpha}_k (\langle x,y\rangle )d\mu(x) d\mu(y) \geq  0, \quad \forall k=0,1,2,...
 \end{equation*}
We can then estimate
\begin{align*}
\int_{S^{d-1}} \int_{S^{d-1}} |\langle x,y\rangle|^p d\mu(x) d\mu(y) & = \int_{S^{d-1}} \int_{S^{d-1}}\sum_{k=0}^p \lambda_k C^{\alpha}_k (\langle x,y\rangle )d\mu(x) d\mu(y)\\
& = \sum_{k=0}^p \lambda_k \int_{S^{d-1}} \int_{S^{d-1}}C^{\alpha}_k (\langle x,y\rangle ) d\mu(x) d\mu(y) \geq \lambda_0.
\end{align*}
From the results in \cite{Seidel:2001aa}, one can deduce that 
 \begin{equation*}
\lambda_0=  \frac{1\cdot 3\cdot 5\cdots(2t-1)}{d(d+2)\cdots (d+2t-2) },
 \end{equation*}
which provides the desired estimate.

We still have to address the ``if and only if'' part. Equality holds if and only if $\mu$ satisfies 
 \begin{equation*}
  \int_{S^{d-1}}\int_{S^{d-1}} C^{\alpha}_k (\langle x,y\rangle ) d\mu(x) d\mu(y) = 0, \quad \forall k=1,\ldots, p. 
 \end{equation*}
We shall follow the approach outlined in \cite{Venkov:2001aa} in which the analog  of Theorem~\ref{theorem:p even integer discrete} was addressed  for finite symmetric collections of points. In this case, the finite symmetric sets of points lead to finite sums rather than integrals as above. The key ideas that we need in order to use the approach presented in \cite{Venkov:2001aa} are: First, $\tilde{\mu}(E):=\frac{1}{2}(\mu(E)+\mu(-E))$, for $E\in\mathcal{B}$, satisfies $\PFP(\tilde{\mu},p) = \PFP(\mu,p)$. Thus, we can assume that $\mu$ is symmetric. Secondly and more critically, the map 
\begin{equation*}
y\mapsto \int_{S^{d-1}} |\langle x,y\rangle |^p d\mu(x)
\end{equation*}
is a polynomial in $y$. In fact, the integral resolves in the polynomial's coefficients. These two observations enable us to follow the lines in \cite{Venkov:2001aa}, and we can conclude the proof.
\end{proof}
 
 \begin{remark}
One may speculate that Theorem \ref{theorem:p even integer} could be extended to $p\geq 2$ that are not even integers. This is not true in general. For $d=2$ and $p=3$, for instance, the equiangular FUNTF with $3$ elements induces a smaller potential than the uniform distribution. The uniform distribution is a probabilistic tight $3$-frame, but the equiangular FUNTF is not.
\end{remark}

\section*{Acknowledgements}
The authors would like to thank C.~Bachoc, W.~Czaja, C.~Wickman, and W.~S.~Yu for discussions leading to some of the results presented here. M.~Ehler was supported by the Intramural Research Program of the National Institute of Child Health and Human Development and by NIH/DFG Research Career Transition Awards Program (EH 405/1-1/575910).  K.~A.~Okoudjou was partially supported by ONR grant N000140910324, by RASA from the Graduate School of UMCP, and by the Alexander von Humboldt foundation.


\bibliographystyle{plain}

\end{document}